\documentclass[reqno, 12pt]{amsart}

\pdfoutput=1

\usepackage{latexsym}
\usepackage[centertags]{amsmath}
\usepackage{amsfonts}
\usepackage{amssymb}
\usepackage{amsthm}
\usepackage{newlfont}
\usepackage{graphics}
\usepackage{color}
\usepackage{float}
\usepackage{diagbox}
\usepackage{longtable}
\usepackage{rotating}
\usepackage{multirow}

\usepackage{extarrows}

\usepackage[sort,compress,numbers]{natbib}

\newtheorem{thm}{Theorem}%[section]
\newtheorem{cor}[thm]{Corollary}
\newtheorem{lem}[thm]{Lemma}

\theoremstyle{mydefinition}

\theoremstyle{myremark}

\allowdisplaybreaks[1]

\def\pa[1]{\frac{\partial}{\partial x}}

\newcommand{\qfac}[1]{(q)_n}

\title{ a sufficient condition for $(\alpha, \beta)$ Somos $4$ Hankel Determinants}

\author{Ying Wang$^1$ and Zihao Zhang$^{2,*}$}

\address{ $^{1}$School of Mathematics and Statistics, North China University of Water Resources and Electric Power,
 Zhengzhou 450045, PR China
 }

 \address{$^{2}$School of Mathematical Sciences, Capital Normal University,
 Beijing 100048, PR China
 }

\email{$^1$\texttt{wangying2019@ncwu.edu.cn}\ \ \ \& $^2$\texttt{zihao-zhang@foxmail.com}}
\date{\today}

\begin{document}

\maketitle

\section{abstract}
By using Sulanke-Xin
continued fractions method, Xin proposed a recursion system to solve the Somos 4 Hankel determinant conjecture. We find
Xin's recursion system indeed give a sufficient condition for $(\alpha, \beta)$ Somos $4$ sequences. This allows us to prove
4 conjectures of Barry on $(\alpha, \beta)$ Somos $4$ sequences in a unified way.

\section{introduction}

An $(\alpha, \beta)$ Somos 4 sequence \cite{Barrysomos} $s_n$ is a sequence such that
$$s_n =\frac{\alpha s_{n-1}s_{n-3}+ \beta s_{n-2}^2}{s_{n-4}} ,\  n \geq 4,$$
for appropriate initial values. Especially, it reduces to the ordinary Somos 4 sequence when $\alpha=1, \beta=1$.
Somos \cite{Somos} considered the fundamental equation $x=y-y^2=z-z^3$. He observed that expanding $y$ as a series in $z$ gives $y=z+z^2+z^3+3z^4+8z^5+23z^6 +\cdots.$
Let $Q(z)=(y-z)/z^2$, he guessed that the Hankel determinants of $Q$ form the Somos 4 sequence. This is also called the Somos 4 conjecture.
Later, the conjecture had been solved by Xin \cite{Xin}.
As part of the tools that can be used to derive Somos sequences, the (sequence) Hankel transforms have been employed \cite{Chang1, Chang2, Shadow, Hone1, Hankel, Xin}.

Let $A=(a_0,a_1,a_2,\dots)$ be a sequence,
and denote by $A(x)=\sum_{n\geq0}a_nx^n$ its generating function.
Define the Hankel transform (or determinants) of $A(x)$ by
$$\mathcal{H}_n(A)= \mathcal{H}_n(A(x)) =(a_{i+j})_{0\leq i,j\leq n-1}, \quad \text{ and }\quad  H_n(A) =H_n(A(x))= \det( \mathcal{H}_n(A)).$$
If it is clear from the context, we will omit the $A$ or $A(x)$.
Traditionally, $\mathcal{H}_0$ is defined to be the empty matrix and $H_0=1$.

 There are many classical tools of continued fractions for evaluating
Hankel determinants, such as the $J$-fractions (Jacobi continued fractions) in Krattenthaler
\cite{kratt} or Wall \cite{wall} and the $S$-fractions in Jones and
Thron \cite[Theorem~7.2]{cfraction}. Our inspiration is the Xin's recursion system derived from the Sulanke and Xin's quadratic transformation,
which was developed from the continued fraction method of Gessel and Xin \cite{xingessel2}.

Our main contribution is the following Theorem 1, which encompasses the four conjecture about $(\alpha, \beta)$ Somos 4 Hankel determinants in \cite{Barrysomos}.
\begin{thm}
 If
 $$Q_0(x)=\frac{a_0+b_0x}{1+cx+d_0x^2+x^2(-1+f_0x)Q_0(x)},$$
then $\{H_n(Q_0)\}_{n \geq 0}$ is  a $(\alpha, \beta)$ Somos $4$ sequence with parameters
$$\alpha=a_0^2  (c+f_0+f_1)^2,$$ and
$$\beta=- \left( c+f_{{0}}+f_{{1}} \right) ^{2}{a_{{0}}}^{3}-a_{{1}} \left(  \left( f_{{0}}-f_{{1}} \right)  \left( c+f_{{0}}+f_{{1}} \right)-a_{{1}}
 \right){a_{{0}}}^{2}.$$
Here $a_1=a_0-d_0+\frac{b_0}{a_0}(c -\frac{b_0}{a_0})$ and $f_1=-\frac{b_0}{a_0}$. 
\end{thm}

The paper is organized as follows.
In Section 2, we introduce Xin's recursion system. In Section 3, we prove Theorem 1. In Section 4, we prove the four conjectures of Barry.

%\section{Somos sequences}
%In this note, we introduce three types of Somos sequences, which Barry had defined in \cite{Barrysomos}.
%$(\alpha, \beta)$ Somos 4 sequences,$(\alpha, \beta, \gamma)$ Somos 6 sequences and $(\alpha, \beta, \gamma,\delta)$  Somos 8 sequences

%An $(\alpha, \beta, \gamma)$ Somos 6 sequence $a_n$ is a sequence such that
%$$a_n =\frac{\alpha a_{n-1}a_{n-5}+ \beta a_{n-2} a_{n-4}+\gamma a_{n-3}^2}{a_{n-6}} , \ n > 6,$$
%for appropriate initial values. Especially,  $\alpha=1, \beta=1, \gamma=1$ is a Somos 6 sequence.
%
%An $(\alpha, \beta, \gamma,\delta)$  Somos 8 sequence $a_n$ is a sequence such that
%$$a_n =\frac{\alpha a_{n-1}a_{n-7}+ \beta a_{n-2} a_{n-6}+\gamma a_{n-3}a_{n-5}+\delta a_{n-4}^2}{a_{n-8}} , \ n > 8,$$
%for appropriate initial values.  Especially,  $\alpha=1, \beta=1, \gamma=1, \delta=1$ is a Somos 8 sequence.
\section{a system of recurrence}
%Sulanke and Xin’s quadratic transformation
Proposition $4.1$ of \cite{Sulanke-xin} defines a quadratic transformation $\tau$ on certain generating function $F(x)$ of algebraic degree two,
and establishes a simple connection between $H_n(F)$ and $H_{n}(\tau (F))$, namely,
$H_n(F) = a^n H_{n-d-1}(\tau (F))$, where $a$ is a constant and $d$ is a nonnegative integer. See \cite{Sulanke-xin} for more details.
What we need here is the following result.

Xin applied the quadratic transformation $\tau$ to the following $F(x)$, and obtained a series of recursion system.
\begin{lem}\cite{Xin} \label{lem-Xin} Suppose $a\neq 0$. If the generating functions $F (x)$ and $G(x)=\tau(F(x))$ are uniquely defined by
\begin{align*}
&F(x)=\frac{a+bx}{1+cx+dx^2+x^2(e+fx)F(x)},\\
&G(x)=\frac{-\frac{a^3e+a^2d-acb+b^2}{a^2}-\frac{a^4f+ca^3d-c^2a^2b+2cab^2-ba^2d-b^3}{a^3}x}{1+cx-\frac{-2acb+2b^2+a^2d}{a^2}x^2+x^2(-1-\frac{b}{a}x)G(x)},
\end{align*}
then $H_n(F)=a^nH_{n-1}(G)$.
\end{lem}

 To be precise, define $Q_0(x) = Q(x)=\frac{a_0+b_0x}{1+c_0x+d_0x^2+x^2(e_0+f_0x)Q_0(x)}$, and recursively define $Q_{n+1}(x)$ to be the unique power series solution of
\[Q_{n+1}(x)=\frac{a_{n+1}+b_{n+1}x}{1+c_{n+1}x+d_{n+1}x^2+x^2(e_{n+1}+f_{n+1}x)Q_{n+1}(x)}.\]
where
\begin{align}
   & a_{n+1}=-\frac{a_n^3e_n+a_n^2d_n-a_nb_nc_n+b_n^2}{a_n^2}, \label{eq a}\\
   & b_{n+1}=-\frac{a_n^4f_n+c_na_n^3d_n-a_n^2c_n^2b_n+2a_nc_nb_n^2-a_n^2b_nd_n-b_n^3}{a_n^3},\label{eq b}\\
   & c_{n+1}=c_n,\label{eq c}\\
   & d_{n+1}=-\frac{a_n^2d_n-2a_nb_nc_n+2b_n^2}{a_n^2},\label{eq d}\\
   & e_{n+1}=-1,\label{eq e}\\
   & f_{n+1}=-\frac{b_n}{a_n}\label{eq f}.
\end{align}

Iteration of Lemma \ref{lem-Xin} clearly gives
\begin{equation}\label{hnq0a}
 H_n(Q_0)=a_0^na_1^{n-1} \cdots a_{n-1}.
\end{equation}

The recursion system was solved explicitly for solving the Somos 4 Conjecture.
\begin{thm}\label{thm333}{\rm\cite{Xin}}
Suppose $c_n=c,\ e_n=-1$, and $a_n$, $b_n$, $d_n$, $f_n$ satisfy the recursion \eqref{eq a}, \eqref{eq b}, \eqref{eq d}, \eqref{eq f}. Then
\begin{equation}\label{eqan}
  a_{n+2}a_{n+1}+a_{n+1}a_n=2a_0a_1+a_0(2f_1+c)(f_0+c+f_1)-\frac{a_0^2(f_0+c+f_1)^2}{a_{n+1}}.
\end{equation}
\end{thm}

\section{the proof of the Theorem 1}

Now we are ready to prove  Theorem 1.
\begin{proof}[Proof of the Theorem 1]
Theorem 1 is equivalent to the following equation,
\begin{equation}\label{hq0}
  H_n(Q_0)=\frac{\alpha H_{n-1}(Q_0) H_{n-3}(Q_0)+\beta H_{n-2}^2(Q_0)}{H_{n-4}(Q_0)},\ \  n \geq 4.
\end{equation}
By \eqref{hnq0a}, the recursion for $H_n(Q_0)$ is transformed  to that for $a_n$ as follow
\begin{equation}\label{e-rec-a}
  a_{{n}} a_{n-1}a_{n-2}-( {\alpha}+\frac{\beta}{ a_{n-1}})=0.
\end{equation}

Equation (\ref{eqan}), with $n$ replaced by $n-2$, can be rewritten as
%\begin{equation*}
% a_{{n+2}}=\frac{2\,a_{{0}}a_{{1}}+a_{{0}} \left( f
%_{{0}}+f_{{1}}+c \right)  \left( 2\,f_{{1}}+c \right) }{a_{n+1}}-{\frac {{a_{{0}
%}}^{2} \left( f_{{0}}+f_{{1}}+c \right) ^{2}}{a_{{n+1}}^2}}- a_{{n}}.
%\end{equation*}
%the above equation becomes
\begin{equation}\label{eq_an}
 a_{{n}}=\frac{2\,a_{{0}}a_{{1}}+a_{{0}} \left( f
_{{0}}+f_{{1}}+c \right)  \left( 2\,f_{{1}}+c \right) }{a_{n-1}}-{\frac {{a_{{0}
}}^{2} \left( f_{{0}}+f_{{1}}+c \right) ^{2}}{a_{{n-1}}^2}}- a_{{n-2}}.
\end{equation}
Put the above into the left hand side of \eqref{e-rec-a}, we obtain
 \begin{align*}
  a_{{n}} a_{n-1}a_{n-2}- \left({\alpha}+\frac{\beta}{ a_{n-1}}\right)=& \frac1{a_{n-1}}\Big(-\left( f_{{0}}+f_{{1}}+c \right) ^{2}a_{{n-2}}{a_{{0}}}^{2}+ \left(
{c}^{2}+cf_{{0}}+3\,cf_{{1}}+2\,f_{{0}}f_{{1}}+2\,{f_{{1}}}^{2}+2\,a_{
{1}} \right) \\
& a_{{n-1}}a_{{n-2}}a_{{0}}-{a_{{n-2}}}^{2}{a_{{n-1}}}^{2}-\alpha\,a_{{n-1}}-\beta \Big).
 \end{align*}
Denote by $T(n-1)$ the numerator of right-hand side of the above equation.
 Replace $n$ with $n+1$, put in $\alpha=a_0^2  (c+f_0+f_1)^2$ (but keep $\beta$ at this moment), and simplify to obtain
\begin{align*}
 T(n)=&- \left( c+f_{{0}}+f_{{1}} \right)^{2}{a_{{0}}}^{2}a_{{n-1}}+ \left(
{c}^{2}+cf_{{0}}+3\,cf_{{1}}+2\,f_{{0}}f_{{1}}+2\,{f_{{1}}}^{2}+2\,a_{
{1}} \right) a_{{n}}a_{{n-1}}a_{{0}}-{a_{{n-1}}}^{2}{a_{{n}}}^{2}\\
&-\left( c+f_0+f_{{1}} \right) ^{2}{a_{{0}}}^{2}a_{{n}}-\beta.
\end{align*}
We claim that $T(n)=0$ for all $n$, so that Theorem 1 holds.

We prove the claim by induction on $n$. The claim is easily checked to be true for $n=1$ (with $\beta=- \left( c+f_{{0}}+f_{{1}} \right) ^{2}{a_{{0}}}^{3}-a_{{1}} \left(  \left( f_{{0}}-f_{{1}} \right)  \left( c+f_{{0}}+f_{{1}} \right)-a_{{1}}
 \right){a_{{0}}}^{2}$). Assume
the claim holds for $n-1$, i.e., $T(n-1)=0$.

Substitute \eqref{eq_an} into $T(n)$ and simplify. We obtain a fraction whose numerator is
\begin{align*}
   -&a_{n-1}^2  \big({a_{{0}}}^{2}\left( c+f_{{0}}+ {f_{{1}}} \right)^2 \left(  a_{{n-2}}+ a_{{n-1}} \right) - \left(
{c}^{2}+cf_{{0}}+3\,cf_{{1}}+2\,f_{{0}}f_{{1}}+2\,{f_{{1}}}^{2}+2\,a_{
{1}} \right) a_{{n-2}}a_{{n-1}}a_{{0}}\\
+&{a_{{n-2}}}^{2}{a_{{n-1}}}^{2
}+\beta \big)\\
=&a_{n-1}^2T(n-1)\\
=&0.
\end{align*}
Thus the claim follows. 
Furthermore, $a_1$ and $f_1$ can be directly obtained from the \eqref{eq a}, \eqref{eq c} ,\eqref{eq e} and  \eqref{eq f} . 
\end{proof}

%\section{Some Applications of Theorem 1}
%A continued fraction of the form
%$$g(x)=\cfrac{1}{1- a_0 x- \cfrac{b_0 x^2}{1-a_1 x - \cfrac{b_1 x^2}{1- a_2 x-\cdots}}}$$ is called a J-fraction.
%Its Hankel determinant is given by $H_n(g) = \prod_{k=0}^n b_k^{n-k}.$ This is independent of the coefficients $a_n$.
\section{Applied to Barry's four conjectures}

Barry \cite{Barrysomos} listed eight conjectures concerning generalized J-fractions and Somos sequences. Among them, four conjectures are $(\alpha,\beta)$ Somos $4$ sequences. We prove them as corollaries of Theorem 1.

\begin{cor}\cite[Conjecture 2]{Barrysomos} We consider the generalized Jacobi continued fraction
$$g(x)=\frac{1}{1-\frac{1+rx}{1-x} x - sx^2 g(x)}.$$
$\{H_{n}(g)\}_{n\geq 1}$ is  a $(0, s^2(r+s+1)^2)$ Somos $4$ sequence.
\end{cor}
\begin{proof}
  Applying Lemma \ref{lem-Xin} two times gives
 $$g_0(x)=\tau (g(x))={\frac { \left( s+r+1 \right) \left( 1-x \right)   }{ 1-2\,x+(r+2){x}^{2}+{x}^{2}\left(
-1+x \right)g_0(x)}},
$$
%%and
%%$$g_1(x)=\tau (g_0(x))={\frac {s \left( 1-x \right) }{1 -2\,x-r{x}^{2
%%}+{x}^{2} \left( -1+x \right)g_1(x)}}.
%%$$

By Theorem 1, we can verify that $a_0=s+r+1, c=-2, f_0= 1, f_1= 1, a_1=s.$
 Thus $$ \alpha=a_0^2  (c+f_0+f_1)^2=0.$$ and $$\beta=- \left( c+f_{{0}}+f_{{1}} \right) ^{2}{a_{{0}}}^{3}-a_{{1}} \left(  \left( f_{{0}}-f_{{1}} \right)  \left( c+f_{{0}}+f_{{1}} \right)-a_{{1}}
 \right){a_{{0}}}^{2}=s^2(r+s+1)^2.$$
 Thus $\{H_{n}(g_0)\}_{n\geq 0}$ is a $(0, s^2(r+s+1)^2)$ Somos $4$ sequence.

By Lemma  \ref{lem-Xin}, we have $H_n(g)=H_{n-1}(g_0)$.
   So the Conjecture holds true.
\end{proof}

%Compare the above Conjecture with Theorem 1, it has been verified that it is not right to start from $g(x)$,
%but it's right to start from $\tau (g(x))$.
%
%Apply two times transformation $\tau$, we have
% $$g_0(x)=\tau (g(x))={\frac { \left( 1-x \right)  \left( s+r+1 \right) }{ 1-2\,x+(r+2){x}^{2}+{x}^{2}\left(
%-1+x \right)g_0(x)}},
%$$
%and
%$$g_1(x)=\tau (g_0(x))={\frac {s \left( 1-x \right) }{1 -2\,x-r{x}^{2
%}+{x}^{2} \left( -1+x \right)g_1(x)}}.
%$$
%
%Thus let $a_0=s+r+1, c=-2, f_0= 1, f_1= 1, a_1=s,$ we verified that $ \alpha=a_0^2  (c+f_0+f_1)^2=0,$ and $\beta=- \left( c+f_{{0}}+f_{{1}} \right) ^{2}{a_{{0}}}^{3}-a_{{1}} \left(  \left( f_{{0}}-f_{{1}} \right)  \left( c+f_{{0}}+f_{{1}} \right)-a_{{1}}
% \right){a_{{0}}}^{2}=s^2(r+s+1)^2.$  So the Conjecture holds true.

\begin{cor}\cite[Conjecture 3]{Barrysomos}  We consider the generalized Jacobi continued fraction
$$g(x)=\frac{1}{1-\frac{1+rx}{1-x} x - \frac{sx^2}{1-x} g(x)}.$$
$\{H_{n}(g)\}_{n\geq 1}$ is a $(s^2, s^2(r+(r+s)^2))$ Somos $4$ sequence.
\end{cor}
\begin{proof}
Applying Lemma \ref{lem-Xin} gives
 $$g_0(x)=\tau (g(x))=\frac{1+s+r-(1+r)x} {1-2\,x+2\,{x}^{2}+r{x}^{2}-{x}^{2}(1-{x})g_0(x)},
$$
%%and
%%$$g_1(x)=\tau (g_0(x))=\frac{-\frac{s ^{2} \left({r+{s}^{2}+2\,rs+{r}^{2}}\right)} {\left({1+s+r}\right)^3}
%%x+\frac{s \left({1+s+2\,r+{s}^{2}+2\,rs+{r}^{2}}\right)} {\left({1+s+r}\right)^{2}}}
%%{1-2 x-\frac{\left({r+2\,{s}^{2}+2\,rs+2\,{r}^{2}+r{s}^{2}+2\,{r}^{2}s+{r}^{3}}\right)} {1+2\,s+2\,r+{s}^{2}+2\,rs+{r}^{2}}x ^{2}
%%-x ^{2} \left({1-\frac{1+r} {1+s+r}x}\right)g_1(x)}.
%%$$

Applying Theorem 1,  we verified that $a_0=s+r+1, c=-2, f_0= 1, f_1= \frac{1+r} {1+s+r}, \\
a_1=\frac{s \left({1+s+2\,r+{s}^{2}+2\,rs+{r}^{2}}\right)} {\left({1+s+r}\right)^{2}}.$ Thus $$ \alpha=a_0^2  (c+f_0+f_1)^2=s^2,$$ and $$\beta=- \left( c+f_{{0}}+f_{{1}} \right) ^{2}{a_{{0}}}^{3}-a_{{1}} \left(  \left( f_{{0}}-f_{{1}} \right)  \left( c+f_{{0}}+f_{{1}} \right)-a_{{1}}
 \right){a_{{0}}}^{2}=s ^{2}\left({r+{s}^{2}+2\,rs+{r}^{2}}\right).$$
Thus, $\{H_{n}(g_0)\}_{n\geq 1}$ is a $(\alpha, \beta)$ Somos $4$ sequence.

By Lemma \ref{lem-Xin}, we have $H_n(g)=H_{n-1}(g_0)$.
   So the Conjecture holds true.
\end{proof}

%Compare the above Conjecture with Theorem 1, it has been verified that it is not right to start from $g(x)$,
%but it's right to start from $\tau (g(x))$.
%
%Apply two times transformation $\tau$, we have
% $$g_0(x)=\tau (g(x))=\frac{1-x+s+r-rx} {1-2\,x+2\,{x}^{2}+r{x}^{2}-{x}^{2}g+{x}^{3}g},
%$$
%and
%$$g_1(x)=\tau (g_0(x))=\frac{-\frac{s ^{2} \left({r+{s}^{2}+2\,rs+{r}^{2}}\right)} {\left({1+s+r}\right)\left({1+2\,s+2\,r+{s}^{2}+2\,rs +{r}^{2}}\right)}
%x+\frac{s \left({1+s+2\,r+{s}^{2}+2\,rs+{r}^{2}}\right)} {\left({1+s+r}\right)^{2}}}
%{1-2 x-\frac{x ^{2}\left({r+2\,{s}^{2}+2\,rs+2\,{r}^{2}+r{s}^{2}+2\,{r}^{2}s+{r}^{3}}\right)} {1+2\,s+2\,r+{s}^{2}+2\,rs+{r}^{2}}
%-\frac{x ^{2} \left({1-x+s+r-rx}\right)} {1+s+r}g_1}.
%$$
%
%Thus let $a_0=s+r+1, c=-2, f_0= 1, f_1= \frac{1+r} {1+s+r}, a_1=\frac{s \left({1+s+2\,r+{s}^{2}+2\,rs+{r}^{2}}\right)} {\left({1+s+r}\right)^{2}}.$ Applying Theorem 1,  we verified that $ \alpha=a_0^2  (c+f_0+f_1)^2=s^2,$ and $\beta=- \left( c+f_{{0}}+f_{{1}} \right) ^{2}{a_{{0}}}^{3}-a_{{1}} \left(  \left( f_{{0}}-f_{{1}} \right)  \left( c+f_{{0}}+f_{{1}} \right)-a_{{1}}
% \right){a_{{0}}}^{2}=s ^{2}\left({r+{s}^{2}+2\,rs+{r}^{2}}\right).$  So the Conjecture holds true.

\begin{cor}\cite[Conjecture 4]{Barrysomos} We consider the generalized Jacobi continued fraction
$$g(x)=\frac{1}{1-\frac{1+rx}{1-x} x - \frac{1+sx}{1-x} x^2g(x)}.$$
$\{H_{n}(g)\}_{n\geq 1}$ is a $((s+1)^2,(1+r^2-6s-3s^2-r(s^2+2s-3)))$ Somos $4$ sequence.
\end{cor}

\begin{proof}
Applying Lemma \ref{lem-Xin} gives
 $$g_0(x)=\tau (g(x))=\frac{ \left( -r+s-1 \right) x+2 +r}{ 1-2x+\left( r+2 \right) {x}^{2}+ \left( {x}^{3}-{x}^{2} \right) g_0(x)},
$$
%%and
%%$$g_1(x)=\tau (g_0(x))=\frac{{\frac { \left( {r}^{3}s+5\,{r}^{2}s+r{s}^{2}+{s}^{3}-{r}^{2}+10\,rs+5
%%\,{s}^{2}-3\,r+11\,s-1 \right)}{ \left( {r}^{2}+4\,r+4 \right)
%% \left( r+2 \right) }}x+{\frac {{r}^{2}-{s}^{2}+4\,r-2\,s+3}{ \left( r+
%%2 \right) ^{2}}}}
%%{1-2\,x-{\frac { \left( {r}^{3}+4\,{r}^{2}+2\,{s}^{2}+4\,r+4\,s+2 \right) }{{r}^{2}+4\,r+4}}{
%%x}^{2} +{x}^{2} \left( {\frac { \left( r-s+1 \right) {x}}{r+2}}-1 \right) g_1(x)}.
%%$$

 Applying Theorem 1,  we verified that
$a_0=r+2, c=-2, f_0= 1, f_1= \frac{r-s+1}{r+2}, \\
a_1=\frac{r^2-s^2+4r-2s+3}{(r+2)^2}.$ Thus $$ \alpha=a_0^2  (c+f_0+f_1)^2=(s+1)^2,$$ and $$\beta=- \left( c+f_{{0}}+f_{{1}} \right) ^{2}{a_{{0}}}^{3}-a_{{1}} \left(  \left( f_{{0}}-f_{{1}} \right)  \left( c+f_{{0}}+f_{{1}} \right)-a_{{1}}
 \right){a_{{0}}}^{2}=1+r^2-6s-3s^2-r(s^2+2s-3).$$ Therefore, $\{H_{n}(g_0)\}_{n\geq 1}$ is a $(\alpha, \beta)$ Somos $4$ sequence.

By Lemma \ref{lem-Xin}, we have $H_n(g)=H_{n-1}(g_0)$.
   So the Conjecture holds true.
\end{proof}

%Compare the above Conjecture with Theorem 1, it has been verified that it is not right to start from $g(x)$,
%but it's right to start from $\tau (g(x))$.
%
%Apply two times transformation $\tau$, we have
% $$g_0(x)=\tau (g(x))=\frac{ \left( -r+s-1 \right) x+2 +r}{ 1-2x+\left( r+2 \right) {x}^{2}+ \left( {x}^{3}-{x}^{2} \right) g_0(x)},
%$$
%and
%$$g_1(x)=\tau (g_0(x))=\frac{{\frac { \left( {r}^{3}s+5\,{r}^{2}s+r{s}^{2}+{s}^{3}-{r}^{2}+10\,rs+5
%\,{s}^{2}-3\,r+11\,s-1 \right)}{ \left( {r}^{2}+4\,r+4 \right)
% \left( r+2 \right) }}x+{\frac {{r}^{2}-{s}^{2}+4\,r-2\,s+3}{ \left( r+
%2 \right) ^{2}}}}
%{1-2\,x-{\frac { \left( {r}^{3}+4\,{r}^{2}+2\,{s}^{2}+4\,r+4\,s+2 \right) {
%x}^{2}}{{r}^{2}+4\,r+4}} +\left( {\frac { \left( r-s+1 \right) {x}^{3}}{r+2}}-{x}^{2} \right) g_1(x)}.
%$$
%
%Thus let
%$a_0=r+2, c=-2, f_0= 1, f_1= \frac{r-s+1}{r+2}, a_1=\frac{r^2-s^2+4r-2s+3}{(r+2)^2},$
%we verified that $ \alpha=a_0^2  (c+f_0+f_1)^2=(s+1)^2,$ and $\beta=- \left( c+f_{{0}}+f_{{1}} \right) ^{2}{a_{{0}}}^{3}-a_{{1}} \left(  \left( f_{{0}}-f_{{1}} \right)  \left( c+f_{{0}}+f_{{1}} \right)-a_{{1}}
% \right){a_{{0}}}^{2}=1+r^2-6s-3s^2-r(s^2+2s-3).$  So the Conjecture holds true.

%The last two conjectures are encompassed in the following more general conjecture.
 \begin{cor}\cite[Conjecture 5]{Barrysomos} We consider the generalized Jacobi continued fraction
$$g(x)=\frac{1}{1-v\frac{1+rx}{1-x} x - w\frac{1+sx}{1-x} x^2g(x)}.$$
$\{H_{n}(g)\}_{n\geq 1}$ is an $(\alpha, \beta)$ Somos $4$ sequence with parameters
$$\alpha = (s+v)^2 w^2,$$ and
$$\beta=w^2(r^2v^2+w(w+v-v^2)+rv(v+2w)-s^2(v(r+1)+2w)-s((r+1)v^2+w+v(r+1+3w))).$$
\end{cor}

\begin{proof}
  Apply Lemma \ref{lem-Xin}, we have
 $$g_0(x)=\tau (g(x))=\frac{\left( -r{v}^{2}+sw-{v}^{2} \right) x+rv+v+w}{
 1+ \left( -v-1 \right) x+ \left( rv+2\,v \right) {x}^{2}+\left( {x}^{3}-{x}^{2} \right) g_0(x)}
,
$$

Applying Theorem 1,  we verified that
$a_0=r v+v+w, c=-v-1, f_0= 1, \\
f_1= {\frac {r{v}^{2}-sw+{v}^{2}}{rv+v+w}}, a_1={\frac { \left( w \left( {v}^{2}-v \right)  \left( r+1 \right) -{w}^{2
} \left( s+1 \right)  \right)  \left( s+v \right) }{ \left( rv+v+w\right) ^{2}}}+w.$ Thus $$ \alpha=a_0^2  (c+f_0+f_1)^2=(s+v)^2 w^2,$$ and
{\small
\begin{align*}
  \beta &=- \left( c+f_{{0}}+f_{{1}} \right) ^{2}{a_{{0}}}^{3}-a_{{1}} \left(  \left( f_{{0}}-f_{{1}} \right)  \left( c+f_{{0}}+f_{{1}} \right)-a_{{1}}
 \right){a_{{0}}}^{2} \\
 &=w^2(r^2v^2+w(w+v-v^2)+rv(v+2w)-s^2(v(r+1)+2w)-s((r+1)v^2+w+v(r+1+3w))).
\end{align*}
}
Therefore, $\{H_{n}(g_0)\}_{n\geq 1}$ is a $(\alpha, \beta)$ Somos $4$ sequence.

By lemma \ref{lem-Xin}, we have $H_n(g)=H_{n-1}(g_0)$.
   So the Conjecture holds true.
\end{proof}

\textbf{Acknowledgments}

 The authors would like to thank Guoce Xin who suggested to study the Hankel determinants of the Somos sequences, and who gave we many valuable discussions.
Part of this work appears in the M.S. thesis \cite{yangyangyou}.


\begin{thebibliography}{12}
\bibliographystyle{alpha}

%%\bibitem{err}
%%{\"O}. E{\u{g}}ecio{\u{g}}lu, T. Redmond, and C. Ryavec, \emph{From
%%a polynomial {R}iemann hypothesis to alternating sign matrices},
%%Electron. J. Combin. \textbf{8} (2001), no.~1,  R36, 51 pp.

\bibitem{Barrysomos} P.  Barry, Conjectures on Somos $4$, $6$ and $8$ sequences using Riordan arrays and the Catalan numbers, {https://arxiv.org/abs/2211.12637.pdf} (2022).

\bibitem{Chang1} X.-K. Chang, X.-B. Hu, and G. Xin, Hankel determinant solutions to several discrete
integrable systems and the Laurent property, SIAM J. Discrete Math. , \textbf{29} (2015) 667--682.

\bibitem{Chang2}  X.-K. Chang and X.-B. Hu, A conjecture based on Somos-$4$ sequence and its extension,
Linear Algebra Appl. , \textbf{436} (2012), 4285--4295.


\bibitem{xingessel2}
I.~M.~Gessel and G.~Xin, The generating function of ternary trees and
  continued fractions, Electron. J. Combin., \textbf{13} (2006), R53.
%%%July  This reference needs updating.

\bibitem{Shadow} A. N. W. Hone, Casting light on shadow Somos sequences, Glasgow Mathematical Journal, (2022), 1--15.


\bibitem{Hone1} A. N. W. Hone, Continued fractions and Hankel determinants
from hyperelliptic curves, Comm. Pure Appl. Math. , \textbf{74} (2021), 2310--2347.

\bibitem{cfraction}
W.~B.~Jones and W.~J.~Thron, Continued Fractions: Analytic
Theory and Applications,
  Encyclopedia of Mathematics and its Applications, vol.~11, Addison-Wesley, Reading, Mass. , 1980.
%%%\bibitem{KreAires}
%%% G. Kreweras, Aires des chemins surdiagonaux a
%%%{\'e}tapes obliques permises. {\it Cahier du B.U.R.O.} 24 (1976) 9-18.

\bibitem{kratt}
 C.~Krattenthaler,
 Advanced determinant calculus: a complement, Linear Algebra
     Appl.  \textbf{411} (2005), 68--166.

\bibitem{Hankel} J. W. Layman, The Hankel transform and some of its properties, J. Integer Seq. , \textbf{4(1)} (2001), 1--11.


%%%The Andrews Festschrift (Maratea, 1998). \emph{Sem. Lothar. Combin.}
%%%\textbf{42} (1999), Art. B42q, 67 pp. (electronic).
%%
%%%%\bibitem{Pergolaetal} E.~Pergola, R.~Pinzani,   S.~Rinaldi, and R.~A.~Sulanke,
%%%%Lattice Path Moments by Cut and Paste. \emph{Adv. in Appl. Math.}
%%%%\textbf{30} (2003), no. 1-2, 208--218.
%%
%%%\bibitem{Shapiro}  L.~W.~Shapiro,  A Catalan triangle. \emph{Discrete Math.} \textbf{14} no. 1  (1976),  83--90.

\bibitem{Somos}
M.~Somos, {\tt
http://grail.cba.csuohio.edu/\raisebox{-3pt}{\~{}}somos/nwic.html}.

\bibitem{Sulanke-xin}  R.~A.~Sulanke and G.~Xin,
Hankel determinants for some common lattice paths, Adv. in
Appl. Math. , \textbf{40} (2008) 149-167.
%%%% to appear, appeared at Formal Power Series and
%%%%Algebraic Combinatorics (FPSAC06).
%%%
%%%%\bibitem{tamm}
%%%%U. Tamm, \emph{Some aspects of {H}ankel matrices in coding theory
%%%%and combinatorics}, Electron. J. Combin. \textbf{8} (2001), no.~1,
%%%%A1, 31 pp.
%%%
%%%%%\bibitem{Viennot} G.~Viennot,
%%%%%  Une th\'eorie combinatoire des polyn\^omes orthogonaux g\'en\'eraux,
%%%%%  Lecture notes, Univ. Quebec, Montr\'eal, Que., 1983.

\bibitem{wall}
H.~S.~Wall,  Analytic Theory of Continued Fractions,  Van
Nostrand, New York, 1948.

%%%%\bibitem{book} P. Barry, \emph{Riordan Arrays: a Primer}, Logic Press, 2017.
%%%
%%%%\bibitem{GenCat} P. Barry, Generalized Catalan numbers, Hankel transforms and Somos-$4$ sequences, \emph{J. Integer Seq.}, \textbf{13} (2010), {https://cs.uwaterloo.ca/journals/JIS/VOL13/Barry1/barry95r.html} {Article 10.7.2}.
%%%
%%%%\bibitem{Braden} H. W. Braden, V. Z. Enolskii, and A. N. W. Hone, (2005) Bilinear recurrences and addition formulae for hyperelliptic sigma functions. \emph{J. Nonlinear Math. Phys.}, \textbf{12} (2005) 46--62.
%%%
%%%%\bibitem{Laurent}  J.~G.~Bogoutdinova, On the Laurent property of the Somos-$8$ sequences, \emph{Dal'nevost. Mat. Zh.}, \textbf{19} (2019), 6--9.
%%%
%%%
%%%
%%%%\bibitem{DRR} R. De Castro, A. Ram\'irez, and J. L. Ram\'irez, Applications in enumerative combinatorics of
%%%%infinite weighted automata and graphs, \emph{Sci. Ann. Comput. Sci.}, \textbf{24} (2014), 137--171.
%%%
%%%%\bibitem{Fed} Y. N. Fedorov and A. N. W. Hone, Sigma-function solution to the general Somos-$6$ recurrence via hyperelliptic Prym varieties. \emph{J. Integrable Syst.}, \textbf{1} (2016),  xyw012, {https://doi.org/10.1093/integr/xyw012}.
%%%
%%%%\bibitem{Hamad}
%%%%K. Hamad, A. N. W. Hone, P. H. van der Kamp, and G. R. W Quispel, QRT maps and related Laurent systems. \emph{Adv. Appl. Math.}, \textbf{96} (2018), 216--248.
%%%
%%%%\bibitem{Hone2} A. N. W. Hone, Analytic solutions and integrability for bilinear recurrences of order six, \emph{Appl.Anal.}, \textbf{89} (2010), 473-492.
%%%
%%%%\bibitem{Hone3} A. N. W. Hone and C. Swart, Integrality and the Laurent phenomenon for Somos $4$ and Somos $5$ sequences. \emph{Math. Proc. Cambridge Philos. Soc.}, \textbf{145} (2008), 65--85.
%%%
%%%%\bibitem{Hone4} A. N. W. Hone, Sigma function solution of the initial value problem for Somos 5 sequences, \emph{Trans. Am. Math. Soc.}, \textbf{359} (2007), 5019--5034.
%%%
%%%%\bibitem{Hone5} A. N. W. Hone, Elliptic curves and quadratic recurrence sequences, \emph{Bull. Lond. Math. Soc.}, \textbf{37} (2005), 161--171.


%%%%%{https://www.cs.uwaterloo.ca/journals/JIS/VOL4/LAYMAN/hankel.html} .
%%%%
%%%%%\bibitem{bbook} L. Shapiro, R. Sprugnoli, P. Barry, Gi-S. Cheon, T-X. He, D. Merlini, and W. Wang, \emph{The Riordan group and applications}, Springer.
%%%%
%%%%%   \bibitem{SGWW} L. W. Shapiro, S. Getu, W. J. Woan, and L. C. Woodson,
%%%%%The Riordan group, \emph{Discr. Appl. Math.}, \textbf{34} (1991),
%%%%% 229--239.
%%%%
%%%%%\bibitem{SL1} N. J. A.~Sloane, \emph{The
%%%%%On-Line Encyclopedia of Integer Sequences}. Published electronically
%%%%%at \texttt{http://oeis.org}, 2022.
%%%%
%%%%%\bibitem{SL2} N. J. A.~Sloane, The On-Line Encyclopedia of Integer
%%%%%Sequences, \emph{Notices Amer. Math. Soc.}, \textbf{50} (2003),  912--915.
%%%%%
%%%%%\bibitem{Poorten1} A. J. van der Poorten, Hyperelliptic curves, continued fractions, and Somos sequences,
%%%%%\emph{IMS Lecture Notes - Monogr. Ser., Dynamics \& Stochastic}, \textbf{48} (2006), 212--224.
%%%%%
%%%%%\bibitem{Poorten2} A. J. van der Poorten, Elliptic curves and continued fractions, \emph{J. Integer Seq.}, \textbf{8} (2005), {https://cs.uwaterloo.ca/journals/JIS/VOL13/Barry1/barry95r.html} {Article 05.2.5}.
%%%%
%%%%%\bibitem{Wall} H.~S. Wall, \emph{Analytic Theory of
%%%%%    Continued Fractions}, AMS Chelsea Publishing, (2000).



\bibitem{Xin}  G. Xin, Proof of the Somos-$4$ Hankel determinant conjecture,  Adv. Appl. Math. , \textbf{42} (2009), 152-156.

\bibitem{yangyangyou} Y. You,  Somos conjectures and Hankel determinants, M.S. Thesis. Capital Normal University,
2023. (Written in Chinese)


\end{thebibliography}
\end{document}